\documentclass[preprint,review,10pt]{amsart}
\usepackage{amsmath,amssymb,amsfonts,xspace}
\newtheorem{theorem}{Theorem}[section]
\newtheorem{lemma}[theorem]{Lemma}

\theoremstyle{definition}
\newtheorem{definition}[theorem]{Definition}
\newtheorem{example}[theorem]{Example}

\newtheorem{proposition}[theorem]{Proposition}
\newtheorem{corollary}[theorem]{Corollary}

\theoremstyle{remark}
\newtheorem{remark}[theorem]{Remark}

\numberwithin{equation}{section}

\begin{document}

\title{ 2-$nil$ hyperideals of multiplicative hyperrings}

\author{Mahdi Anbarloei}
\address{Department of Mathematics, Faculty of Sciences,
Imam Khomeini International University, Qazvin, Iran.
}

\email{m.anbarloei@sci.ikiu.ac.ir}





\begin{abstract}
In this paper, we introduce and study the notion of 2-$nil$ hyperideals of a multiplicative hyperring. 

\end{abstract}
\maketitle

\section{Introduction}
The study of generalized versions of  prime ideals has been a prominent theme in commutative algebra.   The notion of   2-$nil$ ideals, introduced by Celikel, is one such generalization \cite{Celikel}. 


An  expansion of notion of groups called  hypergroups was proposed by a French mathematician  in 1934
Mathematicians  \cite{marty}.  Since then, this new notion  has   been widely investigated and developed by many authors \cite{f2,f4,f7,f9}.  A hyperoperation $``\circ" $ on non-empty set $H$ is a mapping from $H \times H$ into   the family of all non-empty subsets of $H$. In this case,  $(H,\circ)$ is called hypergroupoid.  Let $H_1,H_2 \subseteq H$ and $h \in H$, then $H_1 \circ H_2 =\cup_{h_1 \in H_1, h_2 \in H_2}h_1 \circ h_2,$ and $ H_1 \circ h=H_1 \circ \{h\}.$ This means that the hyperoperation $``\circ" $ on $H$ can be extended to  subsets of $H$. A hypergroupoid $(H, \circ)$ is called  a semihypergroup if $\cup_{h  \in h_2 \circ h_3}h_1 \circ h=\cup_{h^* \in h_1 \circ h_2} h^* \circ h_3 $ for all $h_1,h_2,h_3 \in H$ which means $\circ$ is associative. A semihypergroup $H$ is called a hypergroup if  $h \circ H=H=H\circ h$ for each  $h \in H$ \cite{f10}.  In 1982, Rota extended classical rings to multiplicative hyperrings in  \cite{f14}. This type of hyperstructurs  has been extensively studied in \cite{   Ghiasvand2, f16, ul}. Recall from \cite{f10} that an algebraic structure $(H,+,\circ)$ is a  commutative multiplicative hyperring if the following
holds: {\bf 1.} $(H,+)$ is an abelian  group,  {\bf 2.} $(H,\circ)$ is a semihypergroup; 
{\bf 3.} $h_1\circ (h_2+h_3) \subseteq h_1\circ h_2+h_1\circ h_3$ and $(h_2+h_3)\circ h_1 \subseteq h_2\circ h_1+h_3\circ h_1$ for each $h_1, h_2, h_3 \in H, $
{\bf 4.}  $h_1\circ (-h_2) = -(h_1\circ h_2)=(-h_1)\circ h_2$ for each $h_1, h_2 \in H$,
{\bf 5.} $h_1 \circ h_2 =h_2 \circ h_1$ for each $h_1,h_2 \in H$.  The multiplicative hyperring $H$ is called strongly distributive if in (3)  the equality holds. 
Assume that $(\mathbb{Z},+,\cdot)$ is the ring of integers. For each subset $\Psi$ of   $\mathbb{Z} $ with $\vert\Psi \vert \geq 2$, there exists a multiplicative hyperring $(\mathbb{Z}_\Psi,+,\circ)$ where $\mathbb{Z}_\Psi=\mathbb{Z}$ and  $h \circ g =\{h.\psi.g\ \vert \ \psi \in \Psi\}$ for every $h,g\in \mathbb{Z}_\Psi$ \cite{das}.
An element $1_H \in H$ is considered as an identity element if $h \in h\circ 1_H$ for every $h \in H$. An element $h \in H$ is called unit, if there exists $g \in H$ such that $e \in h \circ g$.  
Also, a multiplicative hyperring $H$ refers to a hyperfield if each non-zero element of $H$ is unit. A non-empty subset $P$ of  $H$ is a  hyperideal  if {\bf i.}  $p_1 - p_2 \in P$ for every $p_1, p_2 \in P$,  {\bf ii.} $h \circ p \subseteq A$ for every $p \in P$ and $h \in H$ \cite{f10}.
 A proper hyperideal $P$ in  $H$ is a prime hyperideal if $u \circ v \subseteq P$ for $u,v \in H$ implies either  $u \in P$ or $v \in P$ \cite{das}. For any given hyperideal $P$ of the multiplicative hyperring $H$, $\sqrt{P}$ denotes the intersection of all prime hyperideals of $H$ containing  $P$. If  $H$  has no prime hyperideal containing $P$, we assume $\sqrt{P}=H$. Suppose that  $\mathcal{C}(H)= \{h_1 \circ h_2 \circ \cdots \circ h_n \ \vert \ h_i \in H, n \in \mathbb{N}\} \subseteq P^{\ast }(H)$ and  $P$ is a hyperideal of $H$. $P$ refers to a $\mathcal{C}$-hyperideal of $H$ if  $ P \cap X \neq \emptyset$ for all $X \in \mathcal{C}(H)$ implies $X \subseteq P$.  Suppose that the hyperideal 
 $\{h \in H \ \vert \  h^n \subseteq P \ \text{for some} \ n \in \mathbb{N}\} $ is designated by ${\bf r}(P)$. Note that the inclusion ${\bf r}(P)  \subseteq \sqrt{P}$ always holds but other inclusion holds if $P$ is a $\mathcal{C}$-hyperideal of $H$ by Proposition 3.2 in \cite{das}. Furthermore, a hyperideal $P$ of $H$ is  a strong $\mathcal{C}$-hyperideal if $Y \cap P \neq \emptyset$ for each $Y \in \mathfrak{C}(H)$  implies  $Y \subseteq P$ such that $\mathfrak{C}(H)=\{\sum_{i=1}^n X_i \ \vert \ X_i \in \mathcal{C}(H), n \in \mathbb{N}\}$   \cite{phd}. 
  A proper hyperideal $M$ of  $H$ is maximal in $H$ if $M \subsetneq N \subseteq H$ for
each hyperideal $N$ of $H$ implies  $N = H$ \cite{ameri}. The intersection of all maximal hyperideals of $H$ is denoted by $J(H)$. 
 An element $h \in H$ is nilpotent, if   $0 \in h^n $ for some $n \in \mathbb{N}$  \cite{ameri}. 

A proper hyperideal $P$ of $H$ is called
\begin{enumerate}
\item[$\bullet$] 2-absorbing (primary) if $u \circ v \circ w \subseteq P$ for $u,v,w \in H$, then $u \circ v \subseteq P$ or  $u \circ v \subseteq P$ ($u \circ v \subseteq {\bf r}(P)$) or $u \circ v \subseteq P$ ($u \circ v \subseteq {\bf r}(P)$) \cite{Suzen}.
\item[$\bullet$] $n$-hyperideal  if whenever $u,v \in H$ with $u \circ v \subseteq P$ and $u \notin {\bf r}(0)$, then $v \in P$ \cite{Ugurlu}. 
\item[$\bullet$] (2,$n$)-hyperideal if whenever $u \circ v \circ w \subseteq P$ for some $u,v,w \in H$, then $u \circ v \subseteq P$ or   $u \circ v \subseteq {\bf r}(0)$  or    $u \circ v \subseteq {\bf r}(0)$ \cite{Ozel}.
\end{enumerate}

 In this paper, we intend to introduce  the notion  of  2-$nil$ hyperideals  in a commutative multiplicative hyperring.  Among many results in this paper, in Theorem \ref{11} we conclude that   taking the radical preserves the  2-$nil$-property. Proposition \ref{13} demonstrates that every 2-$nil$ hyperideal $P$ properly contained in its radical has the same radical as the zero hyperideal.
 In Theorem \ref{16}, we prove that the number of minimal prime hyperideals containing a 2-$nil$ strong $\mathcal{C}$-hyperideal can not exceed two. Theorem \ref{18} provides an equivalent characterization of the radical of the zero hyperideal in terms of 2-nil and 2-absorbing hyperideals.
 Moreover, several examples and arguments are provided to demonstrate the similarities
 and differences between 2-$nil$  hyperideals and the notions of prime hyperideals, 2-absorbing (primary) hyperideals, $n$-hyperideals and $(2,n)$-hyperideals.


\section{  2- $nil$ hyperideals}
\begin{definition}
Let $H$ be a commutative multiplicative hyperring and let $P$ be a  proper hyperideal  of $H$. We call $P$ a \it{2-$nil$ hyperideal} if whenever $u,v,w \in H$ with $u \circ v \circ w \subseteq P$, then $u \circ v \subseteq {\bf r}(0)$ or $u \circ w \subseteq P$ or $v \circ w \subseteq P$.
\end{definition}
From the definition, it follows that every prime hyperideal is 2-$nil$. Also, $n$-hyperideals are 2-$nil$.  
\begin{example}
\begin{enumerate}
\item Consider  the polynomial ring $\mathbb{Z}[x]$ where  $\mathbb{Z}$ is the ring of integers.  Set, $H=\mathbb{Z}+3x\mathbb{Z}[x]$ and define the hyperoperation $u \circ v =\{2uv,4uv\}$ for every $u,v \in \mathbb{Z}$. Then $P=3x\mathbb{Z}[x]$ is a 2-$nil$ hyperideal of $H$ as  $A=3x\mathbb{Z}[x]$ is a prime hyperideal of $G$ by Example 2.4 in \cite{Ghiasvand2}.

\item Consider the multiplicative hyperring $(\mathbb{Z}_4,+,\circ)$ where for all $u,v \in \mathbb{Z}_4$, $u+v$ is remainder of  $\frac{u+v}{4}$  and $u \circ v =\{\overline{uv},\overline{uv+4}\}$ (mod 4). Then $P=\{0,2\}$ is a 2-$nil$ hyperideal of $\mathbb{Z}_4$ since $P$ is a $n$-hyperideal by Example 5  in \cite{Ozel}.
\end{enumerate}
\end{example} 
Example \ref{K} demonstrates that not every 2-$nil$ hyperideal is prime. 
\begin{example} \label{K}
Consider the multiplicative hyperring $(\mathbb{Z}_6,+,\circ)$ where for all $u,v \in \mathbb{Z}_4$, $u+v$ is remainder of  $\frac{u+v}{6}$  and $u \circ v =\{\overline{uv},\overline{2uv},\overline{3uv}, \overline{4uv}, \overline{5uv}\}$ (mod 6). Then $P=\{0\}$ is a 2-$nil$ hyperideal of $\mathbb{Z}_6$ but $P$ is not prime since $2 \circ 3 \subseteq P$, $2 \notin P$ and $3 \notin P$.
\end{example}
The following example verifies that the concept of $n$-hyperideals and 2-$nil$ hyperideals are different.
\begin{example}
On $\mathbb{Z}$, define $u \circ v=\{uv,2uv,3uv,\ldots\}$. Then $3\mathbb{Z}$ is a 2-$nil$ hyperideal of $\mathbb{Z}$. However, $3\mathbb{Z}$ is not $n$-hyperideal since $6 \circ 5 \subseteq 3\mathbb{Z}$ and $6 \notin {\bf r}(0)$ but $5 \notin 3\mathbb{Z}$. 
\end{example}
\begin{remark} \label{K2}
 Let $P$ be a  hyperideal of $H$ such that $P \subseteq {\bf r}(0)$. If $P$ is a     2-$nil$ hyperideal  of $H$, then $P$ is a (2,$n$)-hyperideal.
\end{remark}
In the next example, we show that the condition $P \subseteq {\bf r}(0)$ in Remark \ref{K2} is crucial.
\begin{example}
On $\mathbb{Z}$, define $u \circ v =\{5uv,7uv\}$. Let $P=5 \mathbb{Z}$. Clearly, $P \nsubseteq {\bf r}(0)$. Although $P$ is a 2-$nil$ hyperideal of $\mathbb{Z}$, it is not a (2,$n$)-hyperideal. Because $1 \circ 1 \circ 5 \subseteq 5\mathbb{Z}$ but $1 \circ 1 \nsubseteq 5\mathbb{Z}$  and $1 \circ 5 \nsubseteq {\bf r}(0)$.
\end{example}
The following theorem establishes that the radical of  a 2-$nil$ $\mathcal{C}$-hyperideal of $H$ is also 2-$nil$.
\begin{theorem} \label{11}
Let $P$ be a proper $\mathcal{C}$-hyperideal of $H$. If $P$ is a 2-$nil$ hyperideal of $H$, then  so is ${\bf r}(P)$. 
\end{theorem}
\begin{proof}
Assume that $P$ is a 2-$nil$ $\mathcal{C}$-hyperideal of $H$. Let $u \circ v \circ w \subseteq {\bf r}(P)$ for some $u,v,w \in H$ such that $u \circ v \nsubseteq {\bf r}(0)$.  Then, there exists some $n \in \mathbb{N}$ such that $u^n \circ v^n \circ w^n \subseteq P$ but $u^n \circ v^n \nsubseteq {\bf r}(0)$. This means that $x \circ y \notin {\bf r}(0)$ for some  $x \in u^n$ and $y \in v^n$. Take any   $z \in w^n$. Therefore, $x \circ y \circ z \subseteq P$.   By the hypothesis, we get $x \circ z \subseteq P$ or $y \circ z \subseteq P$. In the first case, we obtain $u^n \circ w^n \subseteq P$ as $(u^n \circ w^n) \cap     P \neq \emptyset$ and $P$ is a $\mathcal{C}$-hyperideal of $H$. In the second case, we have  $v^n \circ w^n \subseteq P$ as $(v^n \circ w^n )  \cap P \neq \emptyset$ and $P$ is a $\mathcal{C}$-hyperideal of $H$. Hence, $u \circ w \subseteq {\bf r}(P)$ or $v \circ w \subseteq {\bf r}(P)$. Consequently, ${\bf r}(P)$ is a 2-$nil$ hyperideal of $H$.
\end{proof}
The following lemma is useful in the proof of our next result.
\begin{lemma} \label{12}
Every 2-$nil$  hyperideal $P$ of $H$ satisfies ${\bf r}(P)=P \cup {\bf r}(0)$.
\end{lemma}
\begin{proof}
Let $P$ be a 2-$nil$  hyperideal $P$ of $H$. The inclusion $P \cup {\bf r}(0) \subseteq {\bf r}(P)$ always holds. It suffices to prove ${\bf r}(P) \subseteq P \cup {\bf r}(0)$. Let $u \in {\bf r}(P)$. Then, $u^n \subseteq P$ for some $n \in \mathbb{N}$. Take any $x \in u^n$. Since $P$ is a hyperideal of $H$, $x \circ 1_H \circ  1_H    \subseteq P$.  Since  $P$ is a 2-$nil$  hyperideal $P$ of $H$ and $ 1_H \circ1_H \nsubseteq  P$ , we have $x \in x \circ 1_H \subseteq  {\bf r}(0)$. Therefore, $u^n \subseteq {\bf r}(0)$ which means $u \in {\bf r}(0)$. 
\end{proof}
\begin{proposition} \label{13}
Let $P$ be a 2-$nil$ hyperideal $P$ of $H$ such that $P \subsetneq {\bf r}(P)$. Then  ${\bf r}(P)={\bf r}(0).$
\end{proposition}
\begin{proof}
Let  $P$ be a 2-$nil$ hyperideal $P$ of $H$ and $v \in {\bf r}(P) \backslash P$. Take $u \in P$.   If  $u+v \in P$, then $v \in P$ which is impossible. So,  $u+v \notin P$. By Lemma \ref{12},  we get $v,u+v \in   P \cup {\bf r}(0)$ which means $u \in {\bf r}(0)$. Hence, $ {\bf r}(0)$ contains $P$ properly. This implies  ${\bf r}(P)={\bf r}(0)$ by Lemma \ref{12}.
\end{proof}
A hyperring $H$ is   divided if for every prime hyperideal $Q$ of $H$ and for every $u \in H \backslash Q$, we have $Q \subseteq \langle u \rangle$.
\begin{theorem} \label{14}
Let $H$ be a divided hyperring, $P$  a $\mathcal{C}$-hyperideal of $H$ and  ${\bf r}(0)$ a prime hyperideal of $H$ containing $P$. Then  $P$ is a 2-$nil$ hyperideal of $H$ if and only if $P$ is a primary hyperideal. 
\end{theorem}
\begin{proof}
$\Longrightarrow$ Let $u \circ v \subseteq P$ for some $u,v \in H$ such that $u \notin {\bf r}(0)$. From $u \circ v \subseteq {\bf r}(0)$ it follows that $v \in {\bf r}(0)$ as ${\bf r}(0)$ is a prime hyperideal of $H$. By the divided property of $H$, we have ${\bf r}(0) \subseteq \langle u \rangle$. Then there exists $r \in H$ such that $v \in u \circ r$. Since $u \circ v \subseteq u^2 \circ r $,   $(u^2 \circ r) \cap P \neq \emptyset$. So, $u^2 \circ r \subseteq P$ because $P$  is a $\mathcal{C}$-hyperideal of $H$. Since $P$ is a 2-$nil$ hyperideal of $H$ and $u^2 \nsubseteq {\bf r}(0)$, we get $u \circ r \subseteq P$ which means $v \in P$. This implies that $P$ is an $n$-hyperideal.  Thus, $P$ is primary by the proof of Proposition 6 in \cite{Ozel}. 

$\Longleftarrow$ Let $u \circ v \circ w \subseteq P$ for some $u,v,w \in H$ such that $u \circ v \nsubseteq {\bf r}(0)$. Let  $x \in (u \circ v) \backslash   {\bf r}(0)$. Then, $x \circ w \subseteq P$ and $x \notin {\bf r}(P)$. Since  $P$ is a primary hyperideal of $H$, we get $w \in P$ which means $u \circ w \subseteq P$, as needed.
\end{proof}
\begin{remark} \label{15}
Every 2-$nil$ hyperideal   of $H$ is 2-absorbing primary. 
\end{remark}
Note that not every 2-absorbing primary is a 2-$nil$ hyperideal. The following example illustrates this.
\begin{example} \label{K3}
On the ring of integers $\mathbb{Z}$, define $u \circ v=\{2uv,4uv\}$. Then $ 120 \mathbb{Z}$ is a 2-absorbing primary hyperideal of $(\mathbb{Z},+,\circ)$ by Example 3.3 (2) in \cite{Suzen}. Taking  $u=4, v=5, w=6$, we get 
$u \circ v \circ w =\{480, 960,1920\} \subseteq 120 \mathbb{Z}$ but $ u \circ v=\{40,80\} \nsubseteq {\bf r}(0), u \circ w=\{48,96\} \nsubseteq 120 \mathbb{Z}, v \circ w=\{60,120\} \nsubseteq 120 \mathbb{Z}$.

\end{example}
A prime hyperideal $Q$ of $H$ is said to be minimal over $P$ if $P \subseteq Q$  and there is no prime hyperideal $Q^{\prime}$ properly contained in $Q$ that contains $P$. The set of all minimal prime hyperideals over $P$ is denoted by $\text{Min}_H(P)$ \cite{Suzen}.

\begin{theorem} \label{16}
Let $P$ be a strong $\mathcal{C}$-hyperideal $P$ of $H$. If $P$ is a 2-$nil$  hyperideal $P$ of $H$, then  $\vert \text{Min}_H(P)\vert \leq 2.$
\end{theorem}
\begin{proof}
Assume that  $P$ is a 2-$nil$ hyperideal $P$ of $H$. Then, $P$ is a 2-absorbing primay hyperideal by Remark \ref{15}. Therefore, there are at most two prime hyperideals of $H$ that are minimal over $P$ by Thererem 3.8  in \cite{anb4}. 
\end{proof}
\begin{remark} \label{17}
Every 2-absorbing hyperideal of $H$ contained in ${\bf r}(0)$ is 2-$nil$ hyperideal.
\end{remark}
The following example   demonstrates that the 2-absorbing property alone is insufficient to guarantee the 2-$nil$-property. 
\begin{example}
In Example \ref{K3},  let $P=15 \mathbb{Z}$. Then $P$ is a 2-absorbing hyperideal of $\mathbb{Z}$. Taking $u=5, v=3, w=1$, we have $u \circ v \circ w=\{60, 120, 240\}$ but $u \circ v=\{30,60\} \nsubseteq {\bf r}(0)$, $u \circ w=\{10,20\}$ and $v \circ w=\{6,12\} \nsubseteq P$. 
\end{example}
\begin{theorem} \label{18}
Let $\{0\}$ be a $\mathcal{C}$-hyperideal of $H$. Then, the following assertions are equivalent:
\begin{enumerate}
\item ${\bf r}(0)$ is a 2-$nil$ hyperideal of $H$.
\item $\vert \text{Min}_H(0)\vert \leq 2$.
\item ${\bf r}(0)$ is a 2-absorbing hyperideal of $H$.
\end{enumerate}
\end{theorem}
\begin{proof}
(1) $\Longrightarrow$ (2)  Let ${\bf r}(0)$ be a 2-$nil$ hyperideal of $H$. By Theorem \ref{16}, there are at most two prime hyperideals of $H$ that are minimal over ${\bf r}(0)$. Hence, we conclude that $\vert \text{Min}_H(0)\vert \leq 2$.

(2) $\Longrightarrow$ (3) Let $\vert \text{Min}_H(0)\vert \leq 2$ and let $\{0\}$ be a $\mathcal{C}$-hyperideal of $H$. This means that $\text{Min}_H(0)=\{Q\}$   or $ \text{Min}_H(0)=\{Q_1,Q_2\}$. In the first case, ${\bf r}(0)=Q$ is a prime hyperideal and so it is a  2-absorbing hyperideal of $H$. In the second case, ${\bf r}(0)= Q_1 \cap Q_2$ is a 2-absorbing hyperideal of $H$ by Theorem 3.4 in \cite{anb4}. 

(3) $\Longrightarrow$ (1) Since $0 \subseteq {\bf r}(0)$, the claim follows from Remark \ref{17}.
\end{proof}
Recall from \cite{Ozel} that a hyperring $H$ is  reduced if it has no nilpotent elements other than zero.
\begin{theorem} \label{19}
Let $\{0\}$ be a $\mathcal{C}$-hyperideal of $H$. If all 2-$nil$ hyperideals of $H$ are prime, then one of the following holds:
\begin{enumerate}
\item $H$ is reduced.
\item $\vert \text{Min}_H(0) \vert \geq 3$.
\item ${\bf r}(0)$ is a non-maximal prime hyperideal.
\end{enumerate}
\end{theorem}
\begin{proof}
Suppose that $H$ is not a reduced hyperring and $\vert \text{Min}_H(0) \vert \ngeq 3$. Then, $0 \subsetneq {\bf r}(0)$ and we have the following two cases. 

{\bf Case I:} Let $\vert \text{Min}_H(0) \vert =2$. By Theorem \ref{18}, ${\bf r}(0 )$ is a 2-$nil$ hyperideal of $H$. If ${\bf r}(0 )$ is prime, then $\vert \text{Min}_H(0) \vert =1$ which is  impossible. Therefore, ${\bf r}(0 )$ is a 2-$nil$ hyperideal of $H$ that it is not prime,. This is a contradiction.

{\bf Case II:} Let $\vert \text{Min}_H(0) \vert =1$. So,  ${\bf r}(0 )$ is a prime hyperideal of $H$. Now, let ${\bf r}(0 )$ be maximal. Then, every proper hyperideal of $H$ is an $n$-hyperideal by Proposition 5 in \cite{Ozel}. This implies that all proper hyperideals of $H$ are 2-$nil$ hyperideals. By the hypothesis, we conclude that  all proper hyperideals of $H$ are prime. This implies that ${\bf r}(0)=\{0\}$ which is impossible. Thus, ${\bf r}(0)$ is a non-maximal prime hyperideal.
\end{proof}
\begin{theorem} \label{110}
Let $P$ be a proper hyperideal of $H$ and ${\bf r}(0)$   be a  strong $\mathcal{C}$-hyperideal  of $H$. Then, the following assertions are equivalent:
\begin{enumerate}
\item[(i)] $P$ is a 2-$nil$ hyperideal of $H$.
\item[(ii)] If $u \circ v \nsubseteq {\bf r}(0)$ for $u,v \in H$, then $(P: u \circ v) \subseteq (P:u)$ or $(P: u \circ v) \subseteq (P:v)$.
\item[(iii)] If $u \circ v \circ Q \subseteq P$ for $u,v \in H$ and a hyperideal $P$ of $H$, then $u\circ v \subseteq {\bf r}(0)$ or $u \circ Q \subseteq P$ or $v \circ Q \subseteq P$. 
\item[(iv)] If $P_1 \circ P_2 \circ P_3 \subseteq P$ for hyperideals $P_1,P_2,P_3$ of $H$, then $P_1 \circ P_2 \subseteq {\bf r}(0)$ or $P_1 \circ P_3 \subseteq P$ or $P_2 \circ P_3 \subseteq P$.
\end{enumerate}
\end{theorem}
\begin{proof}
(i) $\Longrightarrow$ (ii) Let $P$ be a 2-$nil$ hyperideal of $H$ and $w \in  (P: u \circ v)$  for some $u,v \in H$ such that $u \circ v \nsubseteq {\bf r}(0)$. Therefore, $u \circ v \circ w \subseteq P$ which implies $u \circ w \subseteq P$ or $v \circ w \subseteq P$. Then, we get $w \in (P:u)$ or $w \in  (P:v)$. Thus, $(P: u \circ v) \subseteq (P:u)$ or $(P: u \circ v) \subseteq (P:v)$.

(ii) $\Longrightarrow$ (iii) Let $u \circ v \circ Q \subseteq P$ for $u,v \in H$ and a hyperideal $P$ of $H$ such that  $u\circ v \nsubseteq {\bf r}(0)$, $u \circ Q \nsubseteq P$ and $v \circ Q \nsubseteq P$. Then, $u \circ q_1 \nsubseteq P$ and $v \circ q_2 \nsubseteq P$ for some $q_1, q_2 \in Q$. On the other hand, $u\circ v \circ q_1 \subseteq P$ and $u\circ v \circ q_2 \subseteq P$. Hence, we have $u\circ v \nsubseteq {\bf r}(0)$ but neither  $(P: u \circ v) \subseteq (P:u)$ nor $(P: u \circ v) \subseteq (P:v)$, which is impossible.

(iii) $\Longrightarrow$ (iv) Suppose that $P_1 \circ P_2 \circ P_3 \subseteq P$ for hyperideals $P_1,P_2,P_3$ of $H$  such that $P_1 \circ P_2 \nsubseteq {\bf r}(0)$, $P_1 \circ P_3 \nsubseteq P$ and  $P_2 \circ P_3 \nsubseteq P$. This means that $a_1 \circ b_1 \nsubseteq {\bf r}(0)$, $a_2 \circ P_3 \nsubseteq P$ and $b_2 \circ P_3 \nsubseteq P$ for some $a_1,a_2 \in P_1$ and $b_1,b_2 \in P_2$. Therefore, we have $a_2 \circ b_2 \subseteq {\bf r}(0)$ as $a_2 \circ b_2 \circ P_3 \subseteq P$ but neither $a_2 \circ P_3 \subseteq P$ nor $b_2 \circ P_3 \subseteq P$.  Moreover, $a_1 \circ b_1 \circ P_3 \subseteq P$. Then, we get $a_1 \circ P_3 \subseteq P$ or $b_1 \circ P_3 \subseteq P$ as $a_1 \circ b_1 \nsubseteq {\bf r}(0)$. Then, we consider the following cases:

{\bf Case A:} If $a_1 \circ P_3 \subseteq P$ and $b_1 \circ P_3 \nsubseteq P$, then $a_2 \circ b_1 \subseteq {\bf r}(0)$ as $a_2 \circ b_1 \circ P_3 \subseteq P$ but neither $a_2 \circ P_3 \subseteq P$ nor $b_1 \circ P_3 \subseteq P$. Moreover, $(a_1+a_2) \circ P_3 \subseteq a_1 \circ P_3 +a \circ P_3 \nsubseteq P$ because $a_1 \circ P_3  \subseteq P$ and $a_2 \circ P_3 \nsubseteq P$. Then, we get $(a_1+a_2) \circ b_1 \subseteq {\bf r}(0)$ as $(a_1+a_2)\circ b_1 \circ P_3 \subseteq P$ and $b_1 \circ P_3 \nsubseteq P$. Since ${\bf r}(0)$ is a strong $\mathcal{C}$-hyperideal and $(a_1 \circ b_1 + a_2 \circ b_1) \cap {\bf r}(0) \neq \emptyset$, we obtain $a_1 \circ b_1 + a_2 \circ b_1 \subseteq {\bf r}(0)$ which implies $a_1 \circ b_1 \subseteq {\bf r}(0)$, a contradiction.

{\bf Case B:} Let $b_1 \circ P_3 \subseteq P$ and $a_1 \circ P_3 \nsubseteq P$. Then, by a similar argument to Case I, we get a contradiction.

{\bf Case C:} Assume that  $a_1 \circ P_3 \subseteq P$ and $b_1 \circ P_3 \subseteq P$. Then, $(b_1+b_2) \circ P_3 \subseteq b_1 \circ P_3+b_2 \circ P_3  \nsubseteq P$ as $b_2 \circ P_3 \nsubseteq P$. This implies that $a_2 \circ (b_1+b_2) \subseteq {\bf r}(0)$ as $a_2 \circ (b_1+b_2) \circ P_3 \subseteq P$ and $a_2 \circ P_3 \nsubseteq P$. Since ${\bf r}(0)$ is a strong $\mathcal{C}$-hyperideal and $(a_2 \circ b_1 + a_2 \circ b_2) \cap {\bf r}(0) \neq  \emptyset$, we obtain $a_2 \circ b_1 + a_2 \circ b_2 \subseteq {\bf r}(0)$ which implies $a_2 \circ b_1 \subseteq {\bf r}(0)$ as $a_2 \circ b_2 \subseteq   {\bf r}(0)$. Furthermore,    $(a_1+a_2) \circ P_3 \subseteq a_1 \circ P_3 +a_2 \circ P_3 \nsubseteq P$ as $a_1 \circ P_3  \subseteq P$ but $a_2 \circ P_3 \nsubseteq P$. Using the same reasoning as before, we conclude that $a_1 \circ b_2 \subseteq {\bf r}(0)$. Now, we get $(a_1+a_2) \circ (b_1+b_2) \subseteq {\bf r}(0)$ as $(a_1+a_2) \circ (b_1+b_2) \circ P_3 \subseteq P$ and neither $(a_1+a_2)  \circ P_3 \subseteq P$ nor $ (b_1+b_2) \circ P_3 \subseteq P$. Since ${\bf r}(0)$ is a strong $\mathcal{C}$-hyperideal and $(a_1 \circ b_1+ a_1 \circ b_2+ a_2 \circ b_1 + a_2 \circ b_2) \cap {\bf r}(0) \neq  \emptyset$, we obtain $a_1 \circ b_1+ a_1 \circ b_2+ a_2 \circ b_1 + a_2 \circ b_2 \subseteq {\bf r}(0)$ which implies $a_1 \circ b_1 \subseteq {\bf r}(0)$ as $ a_1 \circ b_2, a_2 \circ b_1, a_2 \circ b_2   \subseteq {\bf r}(0)$. Then we obtain a contradiction.

(iv) $\Longrightarrow$ (i) Assume that $u \circ v \circ w \subseteq P$ for $u,v,w \in H$. Then $\langle u \rangle \circ \langle v \rangle \circ \langle w \rangle \subseteq P$ by the proof of Proposition 3 in \cite{Ozel}. Put $P_1=\langle u \rangle$, $P_2= \langle v \rangle$ and $P_3= \langle w \rangle$. Then, $\langle u \rangle \circ \langle v \rangle \subseteq {\bf r}(0)$ or $\langle u \rangle \circ  \langle w \rangle \subseteq P$ or $  \langle v \rangle \circ \langle w \rangle \subseteq P$. Hence, we have $u \circ v \subseteq {\bf r}(0)$ or $u \circ w \subseteq P$ or $v \circ w \subseteq P$. Consequently, $P$ is a 2-$nil$ hyperideal of $H$.
\end{proof}
A commutative multiplicative  hyperring $H$ is hyperdomain if for all $u,v \in H$, $0 \in u \circ v$ implies that $u=0$ or $v=0$ \cite{f16}. 
\begin{theorem} \label{111}
Let $H$ be a hyperdomain. Then all 2-$nil$ hyperideals of $H$ are prime. 
\end{theorem}
\begin{proof}
Assume that $H$ is a   hyperdomain. Then, ${\bf r}(0)=\{0\}$. Let $P$ be a 2-$nil$ hyperideal  of $H$ that it is not prime. This implies that $u \circ v \subseteq P$ for some $u,v \in H \backslash P$. So, $u \circ v \notin {\bf r}(0)$. From $u \circ v \circ 1_H \subseteq P$ it follows that $u \in u \circ 1_H \subseteq P$ or $v \in v \circ 1_H \subseteq P$, which is a contradiction. Hence, every 2-$nil$ hyperideal  of $H$ is prime.
\end{proof}
A hyperring $H$ is   Noetherian if  every ascending chain $P_1 \subseteq  P_2 \subseteq P_3 \subseteq \ldots$  eventually stabilizes. As an example, all finite hyperrings are Noetherian \cite{Ozel}.
\begin{theorem} \label{112}
Assume that $H$ is a Noetherian hyperring and $0$   is a  strong $\mathcal{C}$-hyperideal  of $H$. If  $\vert \text{Min}_H(0) \vert \geq 3$, then all 2-$nil$ hyperideals of $H$ are prime. 
\end{theorem}
\begin{proof}
 Suppose that $\vert \text{Min}_H(0) \vert = m$ where $m \geq 3$. Then, $\{0\}$ is not a 2-$nil$ hyperideal of $H$ by Theorem \ref{16}. Moreover,  ${\bf r}(0)=\cap_{i=1}^m  Q_i$ such that $Q_1,\ldots, Q_m$ are distinct minimal prime hyperideals of $H$. Set, $A=Q_t \cap Q_s$ for $t \neq s$. Since $ Q_t \circ Q_s \circ H \subseteq A$ but neither $Q_t \circ Q_s \subseteq {\bf r}(0)$ nor $Q_t \circ H \subseteq A$ nor $Q_s \circ H \subseteq A$, we conclude that $A$ is not a 2-$nil$ hyperideal of $H$ by Theorem \ref{110}. Now, put $B=\underbrace{P_t\circ\cdots\circ P_t}_{m\text{ times}}$ where $t \in \{1,2,\ldots,m\}$ and $m>1$. Since $P_t \circ P_t^{m-1} \circ H \subseteq B$ but $P_t \circ P_t^{m-1} \nsubseteq {\bf r}(0)$, $P_t   \circ H \nsubseteq B$ and $  P_t^{m-1} \circ H \nsubseteq B$, $B$ is not a 2-$nil$ hyperideal of $H$. Similarly, one can demonstrate that $P_t \cap P_s^m$ is not a 2-$nil$ hyperideal of $H$. Hence, the 2-$nil$ hyperideals of $H$ coincide with its prime hyperideals. 
\end{proof}


A mapping $\theta$ from the multiplicative hyperring 
$(H_1, +_1, \circ _1)$ into the multiplicative hyperring $(H_2, +_2, \circ _2)$ is a hyperring good homomorphism if $\theta(u +_1 v) =\theta(u)+_2 \theta(v)$ and $\theta(u \circ_1 v) = \theta(u)\circ_2 \theta(v)$ for all $u,v \in H_1$ \cite{f10}.

\begin{theorem} \label{homo} 
Assume that $H_1$ and $H_2$ are two multiplicative hyperrings such that $\theta: H_1 \longrightarrow H_2$ is a hyperring
good homomorphism. Then the following statements hold:
\begin{itemize}
\item[\rm{(i)}]~ If $P_2$ is a 2-$nil$ hyperideal of $H_2$ and $\theta$ is injective, then $\theta^{-1}(P_2)$ is a 2-$nil$   hyperideal of $H_1$.
\item[\rm{(ii)}]~ If   $P_1$ is a  2-$nil$ $\mathcal{C}$-hyperideal   of $H_1$  and $\theta$ is surjective with $Ker (\theta) \subseteq P_1$, then $\theta(P_1)$ is a 2-$nil$ $\mathcal{C}$-hyperideal   of $ H_2$.
\end{itemize}
\end{theorem}
\begin{proof}
(i) Let $ u \circ_1 v \circ_1 w  \subseteq \theta^{-1}(P_2)$ for some $u,v,w \in H_1$. Then we get $ \theta(u \circ_1 v \circ_1  w)=\theta(u) \circ_2 \theta(v) \circ_2 \theta(w) \subseteq P_2$ as $\theta$ is a good homomorphism. Since $P_2$ is a 2-$nil$ hyperideal of $H_2$, we obtain $\theta(u \circ_1 v)=\theta(u) \circ_2  \theta(v) \subseteq {\bf r}(0_{H_2})$ or $\theta(u \circ_1 w)=\theta(u) \circ_2  \theta(w) \subseteq P_2$ or $\theta(v \circ_1 w)=\theta(u) \circ_2  \theta(w) \subseteq P_2$. By the hypothesis, we have $\theta^{-1}({\bf r}(0_{H_2})) \subseteq {\bf r}(0_{H_1})$. Therefore,  $u \circ_1 v \subseteq {\bf r}(0_{H_1})$ or $u \circ_1 w \subseteq \theta^{-1}(P_2)$ or $v \circ_1 w \subseteq \theta^{-1}(P_2)$. Consequently, $\theta^{-1}(P_2)$ is a 2-$nil$   hyperideal of $H_1$.

(ii) Since $P_1$ is a   $\mathcal{C}$-hyperideal   of $H_1$, $\theta(P_1)$ is a   $\mathcal{C}$-hyperideal   of $H_2$ by Proposition 2.8 (i) in \cite{Sen}. Now, let $u \circ_2 v \circ_2 w  \subseteq \theta(P_1)$ for $u,v,w \in H_2$. Then there exists $a,b,c \in H_1$ such that $\theta(a)=u$,  $\theta(b)=v$ and $\theta(c)=w$ as $\theta$ is surjective. Therefore, we have  $\theta(a \circ_1 b \circ_1 c)=\theta(a) \circ_2 \theta(b) \circ_2 \theta(c)\subseteq \theta(P_1)$. Now, pick any $s \in a \circ_1 b \circ_1 c$. Then $\theta(s) \in \theta(a \circ_1 b \circ_1 c) \subseteq \theta(P_1)$ and so there exists $t \in P_1$ such that $\theta(s)=\theta(t)$. Then we get $\theta(s-t)=0$ which means $s-t \in Ker (\theta)\subseteq P_1$ and so $s \in P_1$. Therefore $a \circ_1 b \circ_1 c \subseteq P_1$ as $P_1$ is a $\mathcal{C}$-hyperideal. Since $P_1$ is a 2-$nil$ hyperideal of $H_1$, we get   $a \circ_1 b    \subseteq {\bf r}(0_{H_1})$ or $a \circ_1 c    \subseteq P_1$ or $b \circ_1 c    \subseteq P_1$. Hence, $\theta(a \circ_1 b)=u \circ v     \subseteq {\bf r}(0_{H_2})$ or $\theta(a \circ_1 c)=u \circ w     \subseteq \theta(P_1)$ or $\theta(b \circ_1 c)=v \circ w    \subseteq \theta(P_1)$. Thus, $\theta(P_1)$ is a 2-$nil$ $\mathcal{C}$-hyperideal   of $ H_2$.
\end{proof}
Now, we have the following result.
\begin{corollary}
Let $P$ be a proper hyperideal of $H$ containing the hyperideal $I$. If $P$ is a 2-nil hyperideal of $H$, then $P/I$ is a 2-$nil$ hyperideal of $H/I$. 
\end{corollary}
\begin{proof}
Consider the homomorphism $\pi :H \longrightarrow H/I$ defined by $\pi(u)=u+I$. Now, the claim follows from Theorem \ref{homo} (ii) as $\pi$ is a good epimorphism. 
\end{proof}
\section{Future work} 
In \cite{Qaralleh}, the notion of 2-$nil$ primary ideals was proposed by Qaralleh et al. We plan to introduce the notion of 2-$nil$  primary hyperideals of a commutative multiplicative hyperring $H$ as a generalization of 2-$nil$  hyperideals in futur work. 
\begin{definition}
Let $P$ be a proper hyperideal of $H$. Then,  $P$  is called a 2-$nil$  primary hyperideal if $u \circ v \circ w \subseteq P$ for $u,v,w \in H$   implies that $u \circ v \subseteq {\bf r}(0)$ or $u \circ w \subseteq {\bf r}(P)$ or $v \circ w  \subseteq {\bf r}(P)$.
\end{definition}










\begin{thebibliography}{99}
\bibitem{ameri}
R. Ameri, A. Kordi, S. Hoskova-Mayerova, Multiplicative hyperring of fractions and coprime hyperideals, {\it An. St. Univ. ovidious Constanta,}  {\bf 25}(1) (2017) 5-23.









\bibitem{anb4}
M. Anbarloei, A generalization of prime hyperideals, {\it Journal of Algebraic Systems}, {\bf 8} (1) (2020) 113-127.




















\bibitem{Celikel}
E.Y. Celikel, 2-$nil$ ideals of commutative rings, {\it Bull. Belg. Math. Soc. Simon Stevin,} {\bf  28} (2021)  295-304.



\bibitem{Ciampi}
R.P. Ciampi, R. Rota, Polynomials over multiplicative hyperrings, {\it J.
Discrete Math. Sci. Cryptogr.}, {\bf 6} (2-3) (2003) 217–225.




\bibitem{f2}
P. Corsini, V. Leoreanu, Applications of hyperstructures theory, {\it Adv. Math.,} Kluwer Academic Publishers, 2003.


\bibitem{f4}
I. Cristea, S. Jancic-Rasovic, Compositions hyperrings, {\it An. Stiint. Univ. Ovidius Constanta Ser. Mat.,} {\bf 21}(2) (2013) 81-94.


\bibitem{das}
U. Dasgupta, On prime and primary hyperideals of a multiplicative hyperring, {\it Annals of the Alexandru Ioan Cuza
University-Mathematics,} LVIII (1) (2012) 19-36.

\bibitem{phd} 
U. Dasgupta, On certain classes of hypersemirings, PhD Thsis, University of
Calcutta, 2012.




\bibitem{f10}
B. Davvaz, V. Leoreanu-Fotea, Hyperring theory and applications, {\it International Academic
Press}, USA, 2007.
  

 

 
 
 
 
 

 
 


\bibitem{marty}
F. Marty, Sur une generalization de la notion de groupe, In: {\it 8iem Congres des Mathematiciens Scandinaves, Stockholm}, (1934) 45-49.

\bibitem{f7}
C. G. Massouros, On the theory of hyperrings and hyperfields, {\it Algebra i Logika,} {\bf 24} (1985) 728-742.




\bibitem{Ghiasvand2}
P. Ghiasvand, F. Farzalipour, 1-Absorbing prime avoidance theorem  in multiplicative hyperrings, {\it  Palestine Journal of Mathematics,} {\bf 12}(1) (2023) 900-908.






\bibitem{Ozel}
E. Özel Ay, G. Yeşilot,  On (2,$n$)-hyperideals of commutative multiplicative hyperrings, (2026) https://doi.org/10.1142/S0219498828500272



\bibitem{f16}
R. Procesi, R. Rota, On some classes of hyperstructures, {\it Combinatorics
Discrete Math.,} 208/209 (1999) 485-497.

\bibitem{Qaralleh}
I. Qaralleh, S. Al-Kaseasbeh, E.Y. Celikel, On 2-$nil$ primary ideals of commutative rings, {\it An. St. Univ. Ovidius Constanta,}{\bf 34}(1) (2026) 195-208.

\bibitem{f14}
R. Rota, Sugli iperanelli moltiplicativi, {\it Rend. Di Math., Series} VII, {\bf 4} (1982) 711-724.

\bibitem{Sen}
E. Sengelen Sevim, B.A. Ersoy,  B. Davvaz, Primary hyperideals of multiplicative hyperrings, {\it Eurasian Bulletin of Mathematics}, {\bf 1}(1) (2018) 43-49.

\bibitem{Suzen}
N. Süzen, G. Yesilot, On 2-absorbing primary hyperideals of multiplicative hyperrings. arXiv preprint arXiv:1803.09921.

\bibitem{ul}
G. Ulucak, On expansions of prime and 2-absorbing hyperideals in multiplicative hyperrings, {\it Turkish Journal of Mathematics,} {\bf 43} (2019) 1504-1517.

\bibitem{Ugurlu}
E. A. Ugurlu,   A note on $n$-hyperideals of hyperrings, {\it   Southeast Asian Bulletin of Mathematical Sciences,} {\bf 48}(3) (2024) 311-320.

\bibitem{f9}
T. Vougiouklis, Hyperstructures and their representations, Hadronic Press, Inc., Palm Harber, USA, 1994.





\end{thebibliography}
\end{document}